\documentclass[11pt]{article}

\usepackage{amsmath}
\usepackage{amsfonts}
\usepackage{amssymb}
\usepackage{amsthm}
\usepackage{enumitem}
\usepackage{mathtools}
\usepackage{setspace}
\usepackage{etoolbox}

\newtheoremstyle{mystyle}
{11pt}                          % above space 
{11pt}                          % below space 
{}                                      % body font 
{}                                      % indent amount 
{\bfseries}                     % head font 
{}                                      % post head punctuation 
{5.5pt}                         % space between head and body
{}                                      % head spec

\theoremstyle{mystyle}
\newtheorem{theorem}{Theorem}[section]

\newtheorem{corollary}[theorem]{Corollary}
\newtheorem{example}[theorem]{Example}

\renewenvironment{proof}[1][Proof.]{\vspace{-16.5pt} \begin{trivlist}
        \item[\hskip \labelsep {\bfseries #1}]}{\qed \end{trivlist}}

\setitemize{noitemsep, topsep=5.5pt, parsep=5.5pt, partopsep=0pt}

\allowdisplaybreaks
\appto\normalsize{
        \abovedisplayskip=5.5pt plus 2pt minus 2pt
        \belowdisplayskip=5.5pt plus 2pt minus 2pt
        \abovedisplayshortskip=5.5pt plus 2pt minus 2pt
        \belowdisplayshortskip=5.5pt plus 2pt minus 2pt}
\appto\small{
        \abovedisplayskip=5.5pt plus 2pt minus 2pt
        \belowdisplayskip=5.5pt plus 2pt minus 2pt
        \abovedisplayshortskip=5.5pt plus 2pt minus 2pt
        \belowdisplayshortskip=5.5pt plus 2pt minus 2pt}

\newcommand{\gap}{\vspace{11pt}}

\newcommand{\diag}{\operatorname{Diag}}
\newcommand{\tr}{\operatorname{tr}}

\newcommand{\R}{\mathcal{R}}
\newcommand{\Rn}{\mathcal{R}^n}

\newcommand{\Sn}{\mathcal{S}^n}
\newcommand{\Mn}{\mathcal{M}^n}

\newcommand{\Hn}{\mathcal{H}^n}
\newcommand{\V}{{\cal V}}

\newcommand{\one}{{\bf 1}}

\title{\bf  Some  majorization inequalities induced by   \\ 
Schur products in Euclidean Jordan algebras}
\author{
        M. Seetharama Gowda\\
        Department of Mathematics and Statistics\\
        University of Maryland, Baltimore County\\
        Baltimore, Maryland 21250, USA\\
        gowda@umbc.edu
}

\date{\today}

\begin{document}

\maketitle

\begin{abstract}
In an  Euclidean Jordan algebra $\V$ of rank $n$, an element $x$ is said to be majorized by an element $y$, symbolically $x\prec y$, if the corresponding eigenvalue vector $\lambda(x)$  is majorized by $\lambda(y)$ in $\Rn$.
In this article, we describe pointwise majorization inequalities of the form $T(x)\prec S(x)$, where $T$ and $S$ are linear transformations induced by Schur products. Specializing, we recover analogs of majorization inequalities of Schur, Hadamard, and Oppenheimer
 stated in the setting of Euclidean Jordan algebras, as well as majorization inequalities connecting quadratic and Lyapunov transformations on $\V$.
We also show how Schur products induced by certain scalar means (such as arithmetic, geometric, harmonic, and logarithmic means) naturally lead to majorization inequalities.
\end{abstract}

\vspace{1cm}
\\

\newpage

%%%%%%%%%%%%%%%%%%%%%%%%%%%%%%%%%%%%%%%%%%%%%%%%%%%%
\section{Introduction}
The concept of majorization in $\Rn$ is well-known  in  matrix theory, probability and statistics, 
numerical analysis, etc., see e.g., \cite{marshall-olkin}. Given two (column) vectors $p$ and $q$ in $\Rn$, 
$p$ is said to be majorized by $q$, symbolically $p\prec q$,   if 
\begin{equation}\label{weak majorization}
\sum_{i=1}^{k}p^\downarrow_i\leq \sum_{i=1}^{k}q^\downarrow_i
\end{equation}
for all $1\leq k\leq n$ with equality for $k=n$, where $p^\downarrow$ denotes the decreasing rearrangement of $p$, etc.
By a theorem of Hardy, Littlewood, and P\'{o}lya \cite{bhatia1}, $p\prec q$ 
if and only if $p=Dq$, where $D$ is a doubly stochastic matrix, that is, $D$ is a (square) nonnegative matrix with all row and column sums equal to $1$. By a theorem of Birkhoff \cite{bhatia1}, such a $D$ is a convex combination of permutation matrices.
\\

The above concepts and results have been extended to matrices. Let $\Sn$ and $\Hn$ denote, respectively, the spaces of all $n\times n$ real and complex Hermitian matrices, and $\Mn$ denote the space of all $n\times n$ complex matrices. For any matrix $X\in \Hn$, let $\lambda(X)$ denote the vector of eigenvalues of $X$ written in the decreasing order. Then, majorization in $\Hn$ is defined by: $X\prec Y$ if and only if $\lambda(X)\prec \lambda(Y)$ in $\Rn$. This happens if and only if 
there is a linear transformation $\Phi:\Hn\rightarrow \Hn$ (equivalently, $\Phi:\Mn\rightarrow \Mn$) 
which is doubly stochastic (meaning that it is positive, unital, and trace preserving) such that $X=\Phi(Y)$ see, \cite{ando1}, Theorem 7.1,  \cite{zhan}, p. 23, or \cite{bhatia2}, p. 59.
Going beyond matrices, there is an extensive literature on majorization and related topics in such  settings 
as von Neumann algebras, $C^*$-algebras, and Eaton  triples, 
see e.g., \cite{ando2, alberti-uhlmann, ng et al, eaton-perlman}, and references therein. \\

In this article, we focus on majorization in Euclidean Jordan algebras. Let $(\V,\circ,\langle\cdot,\cdot\rangle)$ denote an 
Euclidean Jordan algebra of rank $n$ \cite{faraut-koranyi};
 the algebras $\Sn$ and $\Hn$ are two  primary examples. Letting $\lambda(x)$ denote  the vector of eigenvalues of $x\in \V$ written in the decreasing order, one defines majorization in $\V$:
$$x\prec y \,\,\mbox{in}\,\,\V \Leftrightarrow \lambda(x)\prec \lambda(y)\,\,\mbox{in}\,\,\Rn.$$ 
This concept has been studied in several recent works \cite{tao et al, gowda-positive-map, jeong-gowda}. Replacing permutation matrices by automorphisms of $\V$ and doubly stochastic matrices by doubly stochastic transformations
(see Section 2  for definitions), one can reformulate/extend classical results mentioned earlier to Euclidean Jordan algebras (\cite{gowda-positive-map}, Theorem 6). In particular, if $\Phi:\V\rightarrow \V$ is a doubly stochastic transformation, then
$$\Phi(x)\prec x\,\,\mbox{for all}\,\,x\in \V.$$

In an Euclidean Jordan algebra $\V$, corresponding to an element $a$, the Lyapunov transformation $L_a$ and the quadratic representation $P_a$ are defined by 

\begin{equation}\label{la and pa}
L_a(x):=a\circ x\quad \mbox{and}\quad P_a(x):=2\,a\circ (a\circ x)-a^2\circ x\quad (x\in \V).
\end{equation}
Exploiting the special nature of these transformations that they are induced by `Schur products', in  a recent article \cite{gowda-holder}, the following pointwise
 majorization inequality was proved:
\[ \label{pa majorized by lasquare}
P_a(x)\prec L_{a^2}(x)\,\,\mbox{for all}\,\,x\in \V.
\] 

Motivated by the above result, in this article, we broaden its proof methodology  to describe  pointwise majorization inequalities  of the form 
$T(x)\prec S(x),$
where $T$ and $S$ are linear transformations induced by Schur products.
To elaborate, we fix a Jordan frame $\{e_1,e_2,\ldots, e_n\}$ in $\V$ and consider the corresponding Peirce decomposition of $\V$ and $x\in \V$ \cite{faraut-koranyi}:
$\V=\sum_{1\leq i\leq j\leq n}V_{ij}$ and  $x=\sum_{i\leq j}x_{ij}$, where $x_{ij}\in \V_{ij}$ for all $i\leq j$. Then, for any matrix $A\in \Sn$, we define 
\begin{equation} \label{bullet product}
A\bullet x:=\sum_{i\leq j}a_{ij}\,x_{ij}. 
\end{equation}
The above product -- which we call a {\it Schur product} -- can be regarded as a generalization of the usual 
Schur/Hadamard product of two real matrices. (Since $\V$ is a real vector space, the above product is restricted to real symmetric matrices $A$.) In \cite{gowda-tao-sznajder},  linear transformations of the form  $x\mapsto A\bullet x$ are
 called  Peirce-diagonalizable transformations. 
Both $L_a$ and $P_a$ turn out to be Peirce-diagonalizable transformations under suitable Jordan frames, see Section 2.
The aim of the paper is to describe  results that state/imply inequalities of the form 
\begin{equation}\label{goal inequality}
A\bullet x\prec B\bullet x\,\,\mbox{for all}\,\,x\in \V,
\end{equation}
where $A$ and $B$ are suitable matrices in $\Sn$. In particular, analogous to a matrix theory result of Bapat and Sunder (\cite{bapat-sunder}, Corollary 2), we show that the inequality 
$A\bullet x\prec x$ holds for all $x$ in $\V$ when $A$ is a correlation matrix (i.e., $A$ is a real positive semidefinite matrix  
having all diagonal entries $1$) with the converse statement holding when $\V$ is a simple algebra. As a consequence, we show that 
if $A=[a_{ij}]$ and $B=[b_{ij}]$ belong to $\Sn$ and 
the matrix $\Big [\,\frac{a_{ij}}{b_{ij}}\,\Big ]$ is a correlation matrix, then (\ref{goal inequality}) holds.
Specializing this result, we recover analogs of 
 inequalities of Schur,
Hadamard, and Oppenheimer stated in the setting of Euclidean Jordan algebras as well as the the above mentioned result relating $P_a$ and $L_{a^2}$.  
We also show how Schur products induced by certain scalar means (such as arithmetic, geometric, harmonic, and logarithmic means) naturally lead to majorization inequalities. For example, when $a$ is a positive element in $\V$ (that is, when all eigenvalues of $a$ are positive), 
\begin{equation}\label{long chain}
(L_{a^{-1}})^{-1}(x)\prec P_{\sqrt{a}}(x)\prec  \int_{0}^{1}P_{a^t,a^{1-t}}(x)\,dt\prec L_a(x)\quad\mbox{for all}\,\,x\in \V,
\end{equation}
where $P_{u,v}:=L_u\,L_v+L_v\,L_u-L_{u\circ v}$.
By way of an application, we deduce inequalities of the form 
$$F(A\bullet x)\leq F(B\bullet x)\,\,\mbox{for all}\,\,x\in \V,$$
where 
$F=f\circ \lambda$ (the composition of $f$ and $\lambda$), with $f:\Rn\rightarrow \R$  permutation invariant and convex.
Specializing, we get  eigenvalue inequalities of the form 
$\lambda_{max}(A\bullet x)\leq \lambda_{max}(B\bullet x)$, $\lambda_{min}(B\bullet x)\leq \lambda_{min}(A\bullet x)$ as well as 
norm inequalities of the form $||A\bullet x||_{sp}\leq ||B\bullet x||_{sp}$, where $||\cdot||_{sp}$ is a spectral norm on $\V$ (which arises as the composition of a permutation invariant norm on $\Rn$ and $\lambda$).

The results of  this paper are strongly motivated by the matrix theory works of Bapat and Sunder \cite{bapat-sunder}, Ando \cite{ando1, ando2}, and presentations found in Zhan (\cite{zhan}, Section 2.2) and Bhatia (\cite{bhatia2}, Chapters 4 and 5).
%%%%%%%%%%%%%%%%%%%%%%%%%%%%%%%%%%%%%%%%%%%%%%%%%%%%%%%%%%%%%%%%%%%%%%%%%%%%%%%%%%%%%%%%%%
\section{Preliminaries}
The  Euclidean $n$-space over $\R$ (with usual inner product and norm) is denoted by $\Rn$. Matrices are written in the form  $A=[a_{ij}]$, were $a_{ij}$ denotes the $(i,j)$th entry. The $k\times k$ matrix with every entry $1$ is denoted by ${\bf 1}_{k\times k}$. A real symmetric or complex Hermitian matrix is positive semidefinite if all its eigenvalues are nonnegative 
or, equivalently, the corresponding quadratic form is nonnegative. Given two real/complex matrices $[a_{ij}]$ and $[b_{ij}]$, their (usual) Schur product is the matrix $[a_{ij}b_{ij}]$. \\
 
Throughout this paper, $(\V, \circ,\langle\cdot,\cdot\rangle)$ denotes an Euclidean Jordan algebra of rank 
$n$ and unit element $e$ \cite{faraut-koranyi, gowda-sznajder-tao}
 with $x\circ y$ denoting the Jordan product and $\langle x,y\rangle$ 
denoting the inner product of $x$ and $y$ in $\V$. 
Any Euclidean Jordan algebra is a direct product/sum 
of simple Euclidean Jordan algebras and every simple Euclidean Jordan algebra is isomorphic to one of five algebras, 
three of which are the algebras of $n\times n$ real/complex/quaternion Hermitian matrices. The other two are: the algebra of $3\times 3$ octonion Hermitian matrices and the Jordan spin algebra.
In the algebras $\Sn$ (of all $n\times n$ real symmetric matrices) and $\Hn$ (of all $n\times n$ complex Hermitian matrices), the Jordan product and the inner product are given, respectively, by 
$$X\circ Y:=\frac{XY+YX}{2}\quad\mbox{and}\quad \langle X,Y\rangle:=\tr(XY),$$
where the trace of a real/complex matrix is the sum of its diagonal entries.

According to the {\it spectral decomposition 
theorem} \cite{faraut-koranyi}, any element $x\in \V$ has a decomposition
$$x=x_1e_1+x_2e_2+\cdots+x_ne_n,$$
where the real numbers $x_1,x_2,\ldots, x_n$ are (called) the eigenvalues of $x$ and 
$\{e_1,e_2,\ldots, e_n\}$ is a Jordan frame in $\V$. (An element may have decompositions coming from different Jordan frames, but the eigenvalues remain the same.) Then, $\lambda(x)$ -- called the {\it eigenvalue vector} of $x$ -- is the vector of eigenvalues of $x$ written in the decreasing order. It is known that $\lambda:\V\rightarrow \Rn$ is (Lipschtiz) continuous \cite{baes}. 
An element  $x$ is said to be {\it invertible} if all its eigenvalues are nonzero; such elements form a dense subset of $\V$.

We use the notation $x\geq 0$ ($x> 0$) when all the eigenvalues of $x$ are nonnegative (respectively, positive). 
When $x\geq 0$ with spectral decomposition $x=x_1e_1+x_2e_2+\cdots+x_ne_n$, we define $\sqrt{x}:=\sqrt{x_1}e_1+\sqrt{x_2}e_2+\cdots+\sqrt{x_n}e_n$; more generally, for a nonnegative number $t$, $x^t:=\sum_{i=1}^{n}x_i^t\,e_i$. 

If $x_1,x_2,\ldots, x_n$ are the eigenvalues of  $x\in \V$, we define the {\it trace and determinant} of $x$ by
$$\tr(x):=x_1+x_2+\cdots+x_n\quad \mbox{and}\quad \det(x):=x_1x_2\cdots x_n.$$
For any index $p\in [1,\infty]$, we define the spectral $p$-norm on $\V$ by 
$$||x||_p:=||\lambda(x)||_p\quad (x\in \V),$$
where $||\lambda(x)||_p$ is the usual $p$-norm of the vector $\lambda(x)$ in 
$\Rn$. Because of continuity of $\lambda$,  we see that the 
trace, determinant, and  spectral $p$-norm
 are continuous. 

It is  known that $(x,y)\mapsto \tr(x\circ y)$ defines another inner product on $\V$ that is compatible with the Jordan product. 
{\it Throughout this paper, we assume that the inner product on $\V$ 
is the trace inner product, that is,
$\langle x,y\rangle=\tr(x\circ y).$}
In this inner product, the norm of any primitive element is one and so any Jordan frame
in $\V$ is an orthonormal set. Additionally,
$\tr(x)=\langle x,e\rangle \quad\mbox{for all}\,\,x\in \V.$

\gap

Given a Jordan frame  $\{e_1,e_2,\ldots, e_n\}$, 
we have the Peirce orthogonal decomposition (\cite{faraut-koranyi}, Theorem IV.2.1):
$\V=\sum_{i\leq j}\V_{ij},$
where $\V_{ii}:=\{x\in \V: x\circ e_i=x\}=\R\,e_i$ and for $i<j$, $\V_{ij}:=\{x\in \V: x\circ e_i=\frac{1}{2}x=x\circ e_j\}.$
Then, for any $x\in \V$ we have 
\begin{equation} \label{long peirce decomposition}
x=\sum_{i\leq j}x_{ij}\quad\mbox{with}\quad  x_{ij}\in \V_{ij}.
\end{equation}
Corresponding to this decomposition and any $A\in \Sn$, we define the Schur product $A\bullet x$ by (\ref{bullet product}).
Numerous properties of this product and related transformations  are covered in \cite{gowda-tao-sznajder, sznajder-gowda-moldovan}. One important result is the following.

\begin{theorem} (\cite{sznajder-gowda-moldovan}, Theorem 4 and \cite{gowda-tao-sznajder}, Theorem 3.1) \label{positivity}
{\it For $A\in \Sn$, consider the following statements:
\begin{itemize}
\item [$(a)$] $A$ is positive semidefinite.
\item [$(b)$] The implication $x\geq 0\Rightarrow A\bullet x\geq 0$ holds in $\V$.
\end{itemize}
Then, $(a)\Rightarrow (b)$. The converse holds when $\V$ is simple.
}
\end{theorem}

To see an example, consider the {\it canonical} Jordan frame $\{E_1,E_2,\ldots, E_n\}$ in $\Sn$, where $E_i$ is the diagonal matrix in $\Sn$ with  $1$ in the $(i,i)$ slot and $0$ elsewhere. Then the corresponding Peirce decomposition of any $X\in \Sn$ results in the standard matrix representation $X=[x_{ij}]$ and the Schur product of $A=[a_{ij}]\in \Sn$ with $X$ results in the usual Schur product $A\bullet X=[a_{ij}x_{ij}]$. (Here we have used the
`bullet' notation, reserving the `circle' notation for the 
Jordan product.)
In this setting, Theorem \ref{positivity} includes  the well-known {\it Schur product theorem}:
{\it if $A=[a_{ij}]$ and $X=[x_{ij}]$ in $\Sn$ are positive semidefinite, then so is their  Schur product $[a_{ij}x_{ij}]$.} 
\\

A nonzero element $c$ in $\V$ is an idempotent if $c^2=c$. Corresponding to such an element, we have (also called 
 Peirce decomposition of $\V$) \cite{faraut-koranyi}:
$$\V=\V(c,1)\oplus\V(c,\frac{1}{2})\oplus \V(c,0),$$
where $\V(c,\gamma):=\{x\in \V:x\circ c=\gamma\,x\}$ and $\gamma\in \{0,\frac{1}{2},1\}.$ For any $x\in \V$, this  
yields the decomposition 
\begin{equation}\label{short peirce decomposition}
x=u+v+w,
\end{equation}
where $u\in \V(c,1)$, $v\in \V(c,\frac{1}{2})$, and $w\in \V(c,0).$
To connect this with 
(\ref{long peirce decomposition}), we write the spectral decomposition of $c$ as:
$c=e_1+e_2+\cdots+e_k+0\,e_{k+1}+\cdots+0\,e_n$ for some $k$, $1\leq k\leq n$.  Corresponding to the induced Jordan frame, (\ref{long peirce decomposition}) leads to  
\begin{equation}\label{full peirce decomposition}
x= \sum_{1\leq i\leq j\leq k}x_{ij}\,+\,\sum_{1\leq i\leq k,\, k+1\leq j\leq n}x_{ij}\,+\,\sum_{k+1\leq i\leq j\leq n}x_{ij}.
\end{equation}
This is of the form (\ref{short peirce decomposition}) with  $u=\sum_{1\leq i\leq j\leq k}x_{ij}$, etc. 
Going in the reverse direction, if we have a decomposition of the form (\ref{short peirce decomposition}) coming from an idempotent, then we can find a Jordan frame $\{e_1,e_2,\ldots, e_n\}$ with $c=e_1+e_2+\cdots+e_k$ for some $k$, $1\leq k\leq n$, for which $x$ could be written in the form (\ref{full peirce decomposition}).

For $a\in \V$, we define the transformations $L_a$ and $P_a$ on $\V$ by (\ref{la and pa}). When $\V=\Sn$, with $A$ in place of $a$ and $X$ in place of $x$, these are given by 
$$L_A(X)=\frac{AX+XA}{2}\quad\mbox{and}\quad P_A(X)=AXA.$$
Now suppose $a\in \V$ with its spectral decomposition $a=a_1e_1+a_2e_2+\cdots+a_ne_n$. Let $x=\sum_{i\leq j}x_{ij}$ be the Peirce decomposition of an $x$ with respect to the Jordan frame $\{e_1,e_2,\ldots, e_n\}$. Then, we can describe $L_a$ and $P_a$ in the following form (see \cite{gowda-tao-sznajder}, p. 720):
\begin{equation}\label{schur form of la and pa}
L_a(x)=\sum_{i\leq j}\frac{a_i+a_j}{2}x_{ij}\quad\mbox{and}\quad P_a(x)=\sum_{i\leq j}a_i\,a_j\,x_{ij}.
\end{equation}
So, these two transformations are of the form $x\mapsto A\bullet x$ for suitable $A$.
We mention two well-known properties of $P_a$:
\begin{itemize}
\item [$\bullet$] $x\geq 0\Rightarrow P_a(x)\geq 0$.
\item [$\bullet$] When $a$ is invertible,  $x> 0\Rightarrow P_a(x)> 0$.
\end{itemize}

\gap

A linear transformation $T:\V\rightarrow \V$ is said to be
{\it positive} if $x\geq 0\Rightarrow T(x)\geq 0$, {\it unital\,} if $T(e)=e$, and {\it trace preserving\,} if $\tr(T(x))=\tr(x)$ for all $x$. $T$ is said to be a {\it doubly stochastic transformation} if it satisfies the above three properties.
An invertible linear transformation $T$ on $\V$ is an {\it automorphism} if 
$T(x\circ y)=T(x)\circ T(y)$ for all $x,y\in \V$. 

The following result connects doubly stochastic transformations with majorization in Euclidean Jordan algebras. For other related results, see \cite{gowda-positive-map}.

\begin{theorem}\label{majorization theorem} (\cite{gowda-positive-map}, Theorem 6)
{\it For $x,y\in \V$, consider the following statements:
\begin{itemize}
\item [(a)] $x=\Phi(y)$, where $\Phi$ is a convex combination of automorphisms of $\V$.
\item [(b)] $x=\Phi(y)$, where $\Phi$ is doubly stochastic on $\V$.
\item [(c)] $x\prec y$.
\end{itemize}
Then, $(a)\Rightarrow (b)\Rightarrow(c)$. Furthermore, when $\V$ is $\Rn$  or  simple, the  reverse implications hold.
}
\end{theorem}

We also mention a result of Jeong and Gowda (\cite{jeong-gowda}, Lemma 2)  that says that {\it $\Phi:\V\rightarrow \V$ is doubly stochastic if and only if $\Phi(x)\prec x$ for all $x\in \V$.}

\gap

A function $f:\Rn\rightarrow \R$ is {\it permutation invariant} if $f(\sigma(x))=f(x)$ for all permutation matrices $\sigma$ on $\Rn$. A function $F:\V\rightarrow \R$ is said to be a {\it spectral 
function} \cite{baes, jeong-gowda} if there is a permutation invariant function $f:\Rn\rightarrow \R$ such that $F(x)=f(\lambda(x))$ for all $x$.
In this situation, various properties of $f$ such as continuity, convexity, positive homogeneity, etc., carry over to $F$ \cite{baes, jeong-gowda}. In particular, if $f$ is a (permutation invariant) norm on $\Rn$, then $F$ is a norm on $\V$, which we call a {\it spectral norm}; the spectral $p$-norm, $1\leq p\leq \infty$, is an example. 
Spectral functions have the  automorphism invariance property, namely,  $F\circ T=F$ for any automorphism $T$. 

We end this section by recalling two  convexity results. 
\begin{itemize}
{\it \item [(a)] If $\Omega$ is a convex set in $\Rn$ and $f:\Omega\rightarrow {\cal R}$ is convex, then  \cite{marshall-olkin} 
\begin{equation}\label{schur convexity}
\big [p,q\in \Omega,\,p\prec q\big ] \Rightarrow f(p)\leq f(q).
\end{equation}
\item [(b)] If $f:\Rn\rightarrow \R$ is permutation invariant and convex, then $F:=f\circ \lambda$ is convex on $\V$ \cite{baes} and 
\begin{equation}\label{spectral convexity}
\big [x,y\in \V,\,x\prec y\big ] \Rightarrow F(x)\leq F(y).
\end{equation}
}
\end{itemize}

%%%%%%%%%%%%%%%%%%%%%%%%%%%%%%%%%%%%%%%%%%%%%%%%%%%%%%%%%%%%%%%%%%%%%%%%%%%%%%%%%%%%%%%%%%%%%
\section{Main results and some consequences}

We first describe a  necessary and sufficient condition for $T(x)\prec x$ to hold when $T$ is  a Schur-product induced transformation. This is motivated by a result of Bapat and Sunder (\cite{bapat-sunder}, Corollary 2) proved in the setting of 
$\Hn$.
 
\begin{theorem}\label{first main theorem}
{\it 
We fix a Jordan frame $\{e_1,e_2,\ldots, e_n\}$  relative  to which Schur products are considered. For a given $C\in \Sn$, let  $T(x):=C\bullet x$ on $\V$. 
Consider the  following statements:
\begin{itemize}
\item [(a)] $C$ is a correlation matrix.
\item [(b)] $T$ is doubly stochastic.
\item [(c)] $T(x)\prec x$ for all $x$.
\end{itemize}
Then, $(a)\Rightarrow (b)\Leftrightarrow (c)$. The implication $(c)\Rightarrow (a)$ holds when $\V$ is simple.
}
\end{theorem}

\begin{proof}
$(a)\Rightarrow (b)$: This has been observed in \cite{gowda-positive-map}. For the sake of completeness, we outline its proof. When $(a)$ holds, $C=[c_{ij}]$ is positive semidefinite with $c_{ii}=1$ for all $i$. 
Hence, by Theorem \ref{positivity}, $x\geq 0\Rightarrow T(x)\geq 0$, so $T$ is positive. As the expression $e=\sum_{i=1}^{n} e_i$ can be viewed as the Peirce decomposition of (the unit element) $e$ with respect to the given Jordan frame, $C\bullet e=\sum_{i=1}^{n} c_{ii}e_i=\sum_{i=1}^{n} e_i=e$; hence $T$ is unital.  
Finally, for any $x\in \V$, writing its Peirce decomposition $x=\sum_{i\leq j} x_{ij}$ with respect to $\{e_1,e_2,\ldots, e_n\}$, we have  
$\tr(T(x))=\tr(C\bullet x)=\tr(\sum_{i=1}^{n} c_{ii}x_{ii})=\sum_{i=1}^{n} \tr(x_{ii})=\tr(x)$. Thus, $T$ is trace preserving and so, $T$ is a doubly stochastic.
\\
$(b)\Leftrightarrow (c)$: While the implication $(b)\Rightarrow (c)$ comes from Theorem \ref{majorization theorem}, the equivalence has been observed in \cite{jeong-gowda}, Lemma 2.\\
$(c)\Rightarrow (a)$: Assume that $\V$ is simple and $T(x)\prec x$ for all $x$. This means, by definition, 
$\lambda(T(x))\prec \lambda(x)$ for all $x\in \V$. First, let $x\geq 0$ so that $\lambda(x)$ is a nonnegative vector in $\Rn$. Then, by the Theorem of Hardy, Littlewood, and P\'{o}lya (mentioned in the Introduction), $\lambda(T(x))$ is also nonnegative; hence $C\bullet x=T(x)\geq 0$.
As $\V$ is simple, we invoke Theorem \ref{positivity} to see that $C$ is positive semidefinite.  We now show that 
$c_{11}$, the $(1,1)$ entry of $C$, is one.  By $(c)$, $T(e_1)\prec e_1$, that is, 
$\lambda(T(e_1))\prec \lambda(e_1)$; hence $\tr(C\bullet e_1))=\tr(e_1)$. This simplifies to 
$c_{11}=1$. A similar argument shows that $c_{ii}=1$ for every $i$. Thus, $C$ is a correlation matrix. 
\end{proof}

\gap

\noindent{\bf Remark.} In the setting of the (simple) algebra $\Sn$ with  the canonical Jordan frame,
the implication $(a)\Rightarrow (c)$ is covered by the  result of Bapat and Sunder mentioned earlier. While this result allows Schur products of/by complex Hermitian matrices, in Theorem \ref{first main theorem} we consider products by real symmetric matrices only.  
\\ 
Now, for $n>1$, consider the (non-simple) algebra $\Rn$ and let $C=[c_{ij}]$ be any matrix in $\Sn$ with 
$c_{ii}=1$ for all $i$. Let $\{e_1,e_2,\ldots, e_n\}$ be the standard coordinate system in $\Rn$. Then, up to permutations, this is
is the only Jordan frame in $\Rn$. For any $x\in \Rn$,  the coordinate representation $x=\sum_{i=1}^{n} x_ie_i$ 
can be regarded as a Peirce decomposition of $x$; hence $T(x):=C\bullet x=\sum_{i=1}^{n} c_{ii}x_ie_i=\sum_{i=1}^{n} x_ie_i=x$. Thus, $T(x)\prec x$ holds for all $x$, but $C$ is not necessarily positive semidefinite. So, the implication $(c)\Rightarrow (a)$ need not hold in non-simple algebras.
 
\gap

There are a number   of ways of creating correlation matrices. Some are given below.
\begin{itemize}
\item [$\bullet$] If $A=[a_{ij}]\in \Sn$ is positive semidefinite with positive diagonal entries, then $C=[c_{ij}]$ with
$c_{ij}:=\frac{a_{ij}}{\sqrt{a_{ii}}\sqrt{a_{jj}}}$ is a correlation matrix. 
\item [$\bullet$] Suppose $0<a_1\leq a_2\leq \cdots\leq a_n$ in $\R$. Then, the matrix $A=[a_{ij}]\in \Sn$, where
$a_{ij}:=\frac{a_i}{a_j}$ for $1\leq i\leq j\leq n$, is a correlation matrix, see \cite{bhatia2}, 5.2.22.
\item [$\bullet$] Suppose $a_1,a_2,\ldots, a_n$ are real numbers. Then the matrix $A=[a_{ij}]\in \Sn$, where
$a_{ij}:=\frac{1}{1+|a_i-a_j|}$, is a correlation matrix, see \cite{bhatia2}, 5.2.23.
 \item [$\bullet$] 
Let $\phi$ be a complex-valued function on $\R$ such that for each $N\in \mathbb{N}$ and every choice of real numbers $a_1,a_2,\ldots, a_N$, the $N\times N$ complex matrix $\big [\phi(a_i-a_j)\big ]$ is Hermitian and positive semidefinite. (In \cite{bhatia2}, such functions are called `positive definite'.) 
Clearly, if  such a $\phi$ is real-valued and $\phi(0)=1$, then the real matrix $\big [\phi(a_i-a_j)\big ]$ is symmetric and positive semidefinite, hence a correlation matrix. 
In \cite{bhatia2}, Chapter 5, we find numerous examples of such functions  
 (some of which  are also `infinitely divisible':
$\phi$ is nonnegative and for each $r>0$, the function $\phi^r$ is  `positive definite'). Few examples  are:
$\cos t$, $\frac{\sin t}{t}$, $\frac{t}{\sinh t}$, and $e^{-|t|}$ (see \cite{bhatia2}, 5.2.2, 5.2.6, 5.2.9, and 5.2.17).
\end{itemize}

\gap

We now provide a simple consequence of the above theorem that gives inequalities of the form $A\bullet x\prec B\bullet x$ for all $x$. 
In what follows, given  a matrix $A=[a_{ij}]\in \Sn$ and $k\in \mathbb{N}$, we write
$$A^{(k)}:=[a_{ij}^k]$$
for the (usual) Schur product of $A$ with itself $k$ times. From the Schur product theorem (see Section 2), it follows that if 
$A$ is positive semidefinite, then so is $A^{(k)}$.
When $a_{ij}\neq 0$ for all $i,j$, we define
$$A^{(-k)}:=[a_{ij}^{-k}]\quad (k\in \mathbb{N}).$$

\begin{theorem}\label{second main theorem}
{\it We fix  a Jordan frame in $\V$ relative to which Schur products are considered. Suppose  $A=[a_{ij}], B=[b_{ij}] \in \Sn$
with $b_{ij}\neq 0$ for all $i,j$, and the matrix $\Big [\, \frac{a_{ij}}{b_{ij}}\,\Big ]$ is a correlation 
 matrix.  Let $k\in \mathbb{N}$. 
Then, $A^{(k)}\bullet x\prec B^{(k)}\bullet x\,\,\mbox{for all}\,\,x\in \V$. 
Moreover, if $a_{ij}\neq 0$ for all $i,j$, then,
 $B^{(-k)}\bullet x\prec A^{(-k)}\bullet x\,\,\mbox{for all}\,\,x\in \V$.
}
\end{theorem}

\noindent{\it Note:} In certain settings, via continuity arguments, one can show that the first part of the above conclusion  is valid even when some entries of $B$ are zero.

\gap

\begin{proof} 
Fix $k\in \mathbb{N}$ and $x\in \V$. By assumption, $C:=\Big [\, \frac{a_{ij}}{b_{ij}}\,\Big ]$ is a correlation matrix. In view of Schur product theorem, 
$C^{(k)}=\Big [\,\frac{a_{ij}^{k}}{b_{ij}^{k}}\,\Big ]$ is also a correlation matrix. Then,  
by Theorem \ref{first main theorem}, 
$C^{(k)}\bullet v\prec v$ for all $v\in \V$. For the given $x$,
 we let $v=B^{(k)}\bullet x$ to get 
$$C^{(k)}\bullet (B^{(k)}\bullet x)\prec B^{(k)}\bullet x.$$
This simplifies to $A^{(k)}\bullet x\prec B^{(k)}\bullet x.$ 
\\
Now suppose that $a_{ij}\neq 0$ for all $i,j$. 
Then, the matrix $\Big [\,\frac{b_{ij}^{-k}}{a_{ij}^{-k}}\,\Big ]$ (which equals 
$C^{(k)}$) is a correlation matrix. Hence, by what has been proved earlier, 
$B^{(-k)}\bullet x\prec A^{(-k)}\bullet x.$ 
\end{proof}

\gap

\noindent{\bf Remarks.} Easy examples can be constructed to show that the implication
$T(x)\prec S(x)\Rightarrow S^{-1}(x)\prec T^{-1}(x)$ need not hold even when $T$ and $S$ are invertible.

\gap

\begin{example}\label{example: pa majorized by lasquare}
Consider an element $a\in \V$ with its spectral decomposition $a=a_1e_1+a_2e_2+\cdots+a_ne_n$.
Define matrices $A=[a_{ij}]$ and $B=[b_{ij}]$ with $a_{ij}=a_ia_j$ and
$b_{ij}=\frac{a_i^2+a_j^2}{2}$ for all $i,j$. Then, for any $x$, with respect to the Jordan frame $\{e_1,e_2,\ldots, e_n\}$, 
$P_{a}(x)=A\bullet x$ and $L_{a^2}(x)=B\bullet x$, see (\ref{schur form of la and pa}). 
We claim that for any $x\in \V$ and $k\in \mathbb{N}$,
$A^{(k)}\bullet x\prec B^{(k)}\bullet x$, thus proving the inequality
\begin{equation}\label{pak majorized by lasquarek}
(P_a)^k(x)\prec (L_{a^2})^k(x)\,\,\mbox{for all}\,\,x\in \V,\,k\in \mathbb{N}.
\end{equation}
To see this, we fix $x$ and $k\in \mathbb{N}$, and  assume without loss of generality that $a$ is invertible. 
(The general case follows from a continuity argument involving $\lambda$ and compactness of the set of all $n\times n$ doubly stochastic matrices). Under this invertibility assumption, the matrix
$C:=\big [\frac{a_{ij}}{b_{ij}}\big ]=\big [\frac{2a_ia_j}{a_i^2+a_j^2}\big ]$ is positive semidefinite as it
is the Schur product of the  positive semidefinite matrix $[2a_ia_j]$ and the Cauchy matrix $[\frac{1}{a_i^2+a_j^2}]$ (see \cite{bhatia2}, Exercise 1.1.2). Additionally, the diagonal entries of $C$ are all $1$. Hence we can apply Theorem \ref{second main theorem} to get $A^{(k)}\bullet x\prec B^{(k)}\bullet x$.
\\
We consider two special cases of (\ref{pak majorized by lasquarek}). 
First, when $k=1$, we get the pointwise inequality $P_a(x)\prec L_{a^2}(x)$. 
This implies, in the setting of $\V=\Sn$,
$$AXA\prec \frac{A^2X+XA^2}{2}\,\,\mbox{for all}\,\,A,X\in \Sn.$$ 
For the second, let $k=1$ and $a\geq 0$. Then, by replacing $a$ by $\sqrt{a}$, we get
$$P_{\sqrt{a}}(x)\prec L_a(x)\,\,\mbox{for all}\,\,x\in \V,\,\,a\geq 0.$$
\end{example}

\gap

\begin{corollary}
{\it Consider a Jordan frame $\{e_1,e_2,\ldots, e_n\}$ in $\V$ and let 
$A=[a_{ij}]\in \Sn$ be positive semidefinite. Define $a:=a_{11}e_1+a_{22}e_2+\cdots+a_{nn}e_n$ and $B:=[b_{ij}]$ with $b_{ij}=\sqrt{a_{ii}}\sqrt{a_{jj}}$. 
Then,
\begin{equation} \label{pa inequality}
A\bullet x\prec B\bullet x=P_{\sqrt{a}}(x)\,\,\mbox{for all}\,\,x\in \V.
\end{equation}
Moreover, 
\begin{equation} \label{schur-hadamard inequality}
(a_{11}\,a_{22}\cdots\,a_{nn})\,\det(x)\leq \det(A\bullet x)\,\,\mbox{for all}\,\,x\geq 0\,\,\mbox{in}\, \V.
\end{equation}
}
\end{corollary}

\begin{proof}
First suppose that $A$ is positive definite, in which case all entries of $B$ are nonzero. Then,  it is easy to see that the matrix
$\Big [\, \frac{a_{ij}}{b_{ij}}\,\Big ]$ is a correlation matrix. By the above theorem, $A\bullet x\prec B\bullet x$ for all $x\in \V$. When $A$ is (just) positive semidefinite, we consider $A_k:=A+\frac{1}{k}I$, where $k\in \mathbb{N}$ and $I$ is the identity matrix. Correspondingly, we define $B_k$. For any fixed $x\in \V$, we have
$A_k\bullet x\prec B_k\bullet x$, that is, $\lambda(A_k\bullet x)\prec \lambda(B_k\bullet x)$. Since $\lambda$ is continuous, by using the compactness of the set of all $n\times n$ doubly stochastic matrices, we can let $k\rightarrow \infty$ (via a subsequence) to get
$A\bullet x\prec B\bullet x.$ Since $B\bullet x=\sum_{i\leq j}\sqrt{a_{ii}}\sqrt{a_{jj}}\,x_{ij}=P_{\sqrt{a}}(x)$, we get (\ref{pa inequality}). \\
To see (\ref{schur-hadamard inequality}), we first assume that $A$ is positive definite (so that diagonal entries of $A$ are  positive) and  $x> 0$. We use (\ref{schur convexity}) with (the convex function) $f(p)=-\sum_{i=1}^{n}\ln\, p_i$ (defined on the set of all positive vectors $p$ in $\Rn$) to get
$\det(P_{\sqrt{a}}(x))\leq \det(A\bullet x).$ As   $\det(P_{\sqrt{a}}(x))=\det(a)\det (x)$ (see \cite{faraut-koranyi}, Proposition III.4.2), we get $(a_{11}\,a_{22}\,\cdots\,a_{nn})\,\det(x)\leq \det(A\bullet x).$ By continuity, a similar result holds when $A$ is positive semidefinite and $x\geq 0$.
\end{proof}

Here are some  consequences. 
\begin{itemize}
\item 
Consider $a\in \V$ with its spectral decomposition  $a:=a_{1}e_1+a_{2}e_2+\cdots+a_{n}e_n$ and let $|a|:=|a_{1}|e_1+|a_{2}|e_2+\cdots+|a_{n}|e_n$.  Then,
$$P_a(x)\prec P_{|a|}(x)\quad (x\in \V).$$
To see this, let $A:=[a_{ij}]$ and $B=[b_{ij}]$, where $a_{ij}:=a_ia_j$ and $b_{ij}=\sqrt{a_{ii}}\sqrt{a_{jj}}=|a_i|\,|a_j|$  for all $i,j$. By the above corollary and (\ref{schur form of la and pa}),
$$P_a(x)=A\bullet x\prec B\bullet x=P_{|a|}(x),$$
where the Schur products are defined relative to the Jordan frame $\{e_1,e_2,\ldots, e_n\}$.
\item [$\bullet$] 
{\it If $A\in \Sn$ is positive semidefinite with the diagonal entries positive, then
$x>0\Rightarrow A\bullet x>0 .$
}
This can be seen as follows. 
When the diagonal entries of $A$ are positive and $x>0$,  $a:=a_{11}e_1+a_{22}e_2+\cdots+a_{nn}e_n$ is invertible; hence 
$\sqrt{a}$ is invertible and so  $P_{\sqrt{a}}(x)>0$.
Now, $A\bullet x\prec P_{\sqrt{a}}(x)$ implies that $\lambda(A\bullet x)\prec \lambda(P_{\sqrt{a}}(x))$ in $\Rn$. As the entries of $\lambda(P_{\sqrt{a}}(x))$ are positive, those of $\lambda(A\bullet x)$ must also be positive, implying $A\bullet x>0$.
\item [$\bullet$]
For any $x\in \V$ and $k\in \{1,2,\ldots, n\}$, let $\lambda_{max}(x)$, $\lambda_{min}(x)$, and $S_k(x)$ denote, respectively,
 the largest eigenvalue, smallest eigenvalue, and the sum of the $k$ largest eigenvalues of $x$. It is known (see \cite{baes}, Lemma 20) that as functions of $x$, these are convex and spectral. So, if $A\bullet x\prec B\bullet x$, then, thanks to (\ref{spectral convexity}), 
$$\lambda_{max}(A\bullet x)\leq \lambda_{max}(B\bullet x)\quad\mbox{and}\quad S_k(A\bullet x)\leq S_k(B\bullet x)\quad (x\in \V).$$ 
In this setting, as $\tr(A\bullet x)=\tr(B\bullet x)$, we also have $\lambda_{min}(B\bullet x)\leq \lambda_{min}(A\bullet x)$.
\item [$\bullet$] Suppose $||\cdot||_{sp}$ is a spectral norm on $\V$.
An application of (\ref{spectral convexity}) shows the implication 
$$ A\bullet x\prec B\bullet x\Rightarrow ||A\bullet x||_{sp}\leq ||B\bullet x||_{sp}.$$
To illustrate further,
suppose $A=[a_{ij}]\in \Sn$ is positive semidefinite. By considering the spectral $p$-norm for any $1\leq p\leq \infty$,  from (\ref{pa inequality}) we get 
$||A\bullet x||_p\leq ||P_{\sqrt{a}}(x)||_p$, where  
$a:=a_{11}e_1+a_{22}e_2+\cdots+a_{nn}e_n$. Since 
$||P_b(x)||_p\leq ||b||^{2}_\infty||x||_p$ for any $b\in \V$ (see \cite{gowda-holder}, Theorem 5.3), we have  $||P_{\sqrt{a}}(x)||_p\leq (\max_{i}|a_{ii}|) ||x||_p$ and so
$$||A\bullet x||_p\leq (\max_{i}|a_{ii}|) ||x||_p\quad (x\in \V).$$ 
This is similar to the result given in \cite{bhatia2}, Exercise 2.7.12(iii) for unitary invariant norms on ${\cal M}^n$.
\end{itemize}
\gap
 
We now provide two  examples illustrating the above results. 

\begin{example}
We fix a Jordan frame $\{e_1,e_2,\ldots, e_n\}$ in $\V$ and consider the Peirce decomposition of an $x\in \V$: 
$x=\sum_{i=1}^{n}x_{ii}+\sum_{i<j}x_{ij}.$ Noting  $x_{ii}\in \V_{ii}=\R\,e_i$ we let  
 $$\diag(x):=\sum_{i=1}^{n}x_{ii}$$ denote the ``diagonal" of $x$.
We specialize Theorem \ref{second main theorem} by letting  
 $A$ be the $n\times n$ identity matrix and $B=\one_{n\times n}$. Then, $A\bullet x\prec B\bullet x$ for all $x\in \V$ becomes 
$$\diag(x)\prec x\,\,\mbox{for all}\,\,x\in \V.$$
This is {\it Schur's inequality} in the setting of Euclidean Jordan algebras, see \cite{sznajder-gowda-moldovan}.
This leads to the {\it Hadamard inequality}
$$\det(x)\leq \det(\diag(x)) 
\,\,(\mbox{when}\,\,x\geq 0).$$ 
Combining this with (\ref{schur-hadamard inequality}), one gets the {\it Oppenheimer inequality}: for $x\geq 0$ and any $A\in \Sn$ positive semidefinite, 
$$(a_{11}\,a_{22}\cdots\,a_{nn})\,\det(x)\leq \det(A\bullet x)\leq (a_{11}\,a_{22}\cdots\,a_{nn})\,\det(\diag(x)).$$
This inequality has been observed in \cite{sznajder-gowda-moldovan}, Theorem 5, with a different proof.
\end{example}

\gap

\begin{example}
Consider a nonzero idempotent $c$ in $\V$ and the corresponding Peirce decomposition of any $x$ in the form (\ref{short peirce decomposition}): $x=u+v+w,$
where $u\in \V(c,1)$ and $w\in \V(c,0)$. As in Section 2, we could write $x$ in the form (\ref{full peirce decomposition}) for some Jordan frame. Assuming this, we let $A$ be the $n\times n$ matrix defined by 
$$A=\left [ \begin{array}{cc}
\one_{k\times k} & 0\\
0 & \one_{(n-k)\times (n-k)}
\end{array}\right ].$$
Applying the above corollary, we see that $B=\one_{n\times n}$ and $u+w=A\bullet x\prec B\bullet x=x.$  Thus, 
$$u+w\prec x.$$
This has been observed in \cite{gowda-positive-map} with a different proof  
and was crucially used in proving the H\"{o}lder type inequality \cite{gowda-holder}
$$||x\circ y||_1\leq ||x||_p\,||y||_q,$$
where $1\leq p,q\leq \infty$ with $\frac{1}{p}+\frac{1}{q}=1$.
When $x\geq 0$, a consequence of $u+w\prec x$ is the inequality $\det(x)\leq \det(u+v)=\det(u)\,\det(v)$ , which is known as the {\it generalized Fischer inequality} (see \cite{gowda-tao}, Prop. 2).
\end{example}

\section{Majorization inequalities induced by scalar means}
In Example \ref{example: pa majorized by lasquare}, we proved the inequality $P_{\sqrt{a}}(x)\prec L_{a}(x)$ for  $a\geq 0$ by applying Theorem \ref{second main theorem} to matrices $A=[a_{ij}]$ and $B=[b_{ij}]$ with $a_{ij}=\sqrt{a_ia_j}$ and $b_{ij}=\frac{a_i+a_j}{2}.$
It is interesting to observe that $\sqrt{a_ia_j}$ and $\frac{a_i+a_j}{2}$ are the geometric and arithmetic means of nonnegative numbers $a_i$ and $a_j$. Motivated by this, we consider general (scalar) means and look for possible  majorization and related inequalities.  %Our presentation below is  motivated by \cite{bhatia2}, especially, Chapter 5.%
\\

A   {\it scalar  mean} (in short, mean) is a real valued function $(t,s)\mapsto m(t,s)$ defined on the set of all pairs of positive real numbers satisfying 
a certain set of conditions (see e.g., \cite{bhatia2}, p. 101). For our discussion, we assume that a mean satisfies two basic conditions:
\begin{enumerate}
\item $m(s,t)=m(t,s)>0$ for all $t,s>0$;
\item If $0<t\leq s$, then $t\leq m(t,s)\leq s.$
\end{enumerate}
Examples include the arithmetic, geometric, and harmonic means, given respectively by
$$m_A(t,s)=\frac{t+s}{2},\quad m_G(t,s)=\sqrt{ts},\quad \mbox{and}\quad m_H(t,s)=\Big (\frac{t^{-1}+s^{-1}}{2}\Big )^{-1}.$$
Another is the logarithmic mean given by 
$$m_L(t,s)=\frac{t-s}{\ln t-\ln s}.$$
(The value on the right  is interpreted as $t$ when $s=t.$)
These are related by the inequalities
$$m_H(t,s)\leq m_G(t,s)\leq m_L(t,s)\leq m_A(t,s).$$ 
Now consider  two means $m_1(t,s)$ and $m_2(t,s)$. Given $0<a\in \V$ with its spectral decomposition $a=a_1e_1+a_2e_2+\cdots+a_ne_n$,  we consider $A=[a_{ij}]$ and $B=[b_{ij}]$, where
$$a_{ij}=m_1(a_i,a_j)\quad\mbox{and}\quad b_{ij}=m_2(a_i,a_j).$$ 
As $a_{ii}=a_i=b_{ii}$ for all $i$, we see that the matrix $C:=\big [\frac{a_{ij}}{b_{ij}}\big ]$ has all diagonal entries $1$. So,
when $C$ is positive semidefinite, we can apply our Theorem \ref{second main theorem} to get the pointwise inequality $A\bullet x\prec B\bullet x$. 
In many settings, the positive definiteness of $C$ is established via classical arguments, see e.g., \cite{bhatia2}, Chapter 5.

\begin{example} 
We fix  $0<a\in \V$ with its spectral decomposition 
$a=a_1e_1+a_2e_2+\cdots+a_ne_n$ and consider Schur products relative to the Jordan frame $\{e_1,e_2,\ldots, e_n\}$. 
Given a mean $m$, we use the corresponding capital letter to denote the matrix induced by $m$ and $\{a_1,a_2,\ldots,a_n\}$. 
(For example, in the case of the arithmetic mean $m_A$, $M_A:=[m_A(a_i,a_j)]$.)
Then, we have the following inequalities:
{\it 
\begin{equation}\label{HGA chain1}
M_H\bullet x\prec M_G\bullet x\prec M_A\bullet x\quad\mbox{for all}\,\,x\in \V,
\end{equation}
which translates to
\begin{equation}\label{HGA chain2}
(L_{a^{-1}})^{-1}(x)\prec P_{\sqrt{a}}(x)\prec L_a(x)\quad\mbox{for all}\,\,x\in \V.
\end{equation}
}
Also,
{\it
\begin{equation}\label{GLA chain1}
M_G\bullet x\prec M_L\bullet x\prec M_A\bullet x\quad\mbox{for all}\,\,x\in \V,
\end{equation}
which translates to
\begin{equation}\label{GLA chain2}
P_{\sqrt{a}}(x)\prec \int_{0}^{1}P_{a^t,a^{1-t}}(x)\,dt\prec L_a(x)\quad\mbox{for all}\,\,x\in \V,
\end{equation}
where for any two elements $u,v\in \V$, $P_{u,v}:=L_u\,L_v+L_v\,L_u-L_{u\circ v}.$
}

\gap

In (\ref{HGA chain2}), the second inequality has already been justified. We prove the first inequality.
As 
$$\frac{m_H(a_i,a_j)}{m_G(a_i,a_j)}=\frac{2\sqrt{a_i}\sqrt{a_j}}{a_i+a_j}$$
leads to a correlation matrix (see the proof in Example \ref{example: pa majorized by lasquare}),
we see (by Theorem \ref{second main theorem}) that $M_H\bullet x\prec M_G\bullet x$ for all $x$. Now, 
$$L_{a^{-1}}(x)=\sum_{i\leq j}\frac{a_i^{-1}+a_j^{-1}}{2}\,x_{ij}\quad\mbox{and}\quad P_{\sqrt{a}}(x)=\sum_{i\leq j}\sqrt{a_i}\sqrt{a_j}x_{ij}$$ and so
$$(L_{a^{-1}})^{-1}(x)=  \sum_{i\leq j}\frac{2}{a_i^{-1}+a_j^{-1}}\,x_{ij} =M_H\bullet x\prec M_G\bullet x= P_{\sqrt{a}}(x)\quad (\mbox{for all}\,\,a>0\,\,\mbox{and}\,\,x\in \V).$$
\end{example}

We proceed to prove (\ref{GLA chain1}). 
The matrix with entries
$$\frac{m_G(a_i,a_j)}{m_L(a_i,a_j)}=\frac{\sqrt{a_i}\sqrt{a_j}(\ln a_i-\ln a_j)}{a_i-a_j}$$
is a correlation matrix (see \cite{bhatia2}, p. 164 for its positive semidefiniteness). This, via Theorem \ref{second main theorem}  gives the first inequality in 
(\ref{GLA chain1}). Similarly, 
the matrix with entries
$$\frac{m_L(a_i,a_j)}{m_A(a_i,a_j)}=\frac{2(a_i-a_j)}{(\ln a_i-\ln a_j)(a_i+a_j)}$$
is a correlation matrix (see \cite{bhatia2}, p. 164 for its positive semidefiniteness). This gives the second inequality in
(\ref{GLA chain1}).
To go from (\ref{GLA chain1}) to (\ref{GLA chain2}), we need only to describe $M_L\bullet x.$
Now, using (\ref{schur form of la and pa}) and the definition of $P_{a^t,a^{1-t}}$,
we see that  
$$
P_{a^t,a^{1-t}}(x)=\sum_{i\leq j}\Big (\frac{a_j(\frac{a_i}{a_j})^t+a_i(\frac{a_j}{a_i})^t}{2}\Big )\,x_{ij},$$
which, upon integration and simplification, leads to 
$$\int_{0}^{1}P_{a^t,a^{1-t}}(x)\,dt=M_L\bullet x.$$

Similar to (\ref{GLA chain2}), based on Theorem \ref{second main theorem} and some
 additional work, one can show the following:
\begin{equation} \label{doublep inequality}
P_{\sqrt{a}}(x)\prec P_{a^t,a^{1-t}}(x)\prec  L_a(x)\quad\mbox{for all}\,\,x\in \V, a>0,\, 0\leq t\leq 1.
\end{equation}
In fact, the first inequality can be proved by considering the correlation matrix whose entries are
$$\frac{2\sqrt{a_i}\sqrt{a_j}}{a_i^t a_j^{1-t}+a_j^t a_i^{1-t}}.$$
(The matrix is positive semidefinite as it is the Schur product of  
the  positive semidefinite matrix $\big [ 2 a_i^{t-\frac{1}{2}} a_j^{t-\frac{1}{2}}\big ]$ and the Cauchy matrix 
$\big [ \frac{1}{a_i^{2t-1}+a_j^{2t-1}}\big ]$.)
The second inequality can be proved by considering the correlation matrix whose entries are
$$\frac{a_i^t a_j^{1-t}+a_j^t a_i^{1-t}}{a_i+a_j}.$$
(For its positive semidefiniteness, see \cite{bhatia2}, p. 161.)

We finally note that in the case of $\V=\Sn$, $P_{A,B}(X)=\frac{AXB+BXA}{2}$ (\cite{faraut-koranyi}, p. 32) and so, for positive definite $A$, (\ref{doublep inequality}) becomes 
$$\sqrt{A}X\sqrt{A}\prec \frac{A^tXA^{1-t}+A^{1-t}XA^t}{2} \prec \frac{AX+XA}{2}.$$

\gap

\begin{example}
Suppose $m_1$ and $m_2$ are two means with $m_1(t,s)\leq m_2(t,s)$ for all positive real numbers $t,s>0.$
For any two positive numbers $a_1$ and $a_2$, let $a_{ij}=m_1(a_i,a_j)\quad\mbox{and}\quad b_{ij}=m_2(a_i,a_j)$ where $i=1,2$. Then, in the (rank $2$)  Jordan spin algebra, with respect to any Jordan frame $\{e_1,e_2\}$, we have 
$A\bullet x\prec B\bullet x$. This is because the $2\times 2$ matrix $C:=[\frac{a_{ij}}{b_{ij}}]$ has diagonal entries $1$ and off-diagonal entries less than or equal to one.
\end{example}

\gap

\section{Exponential metric increasing property}
Here, we provide another application of Theorem \ref{second main theorem}. \\
Let $I$ be an open interval in $\R$ and ${\V}_{I}$ denote the set of all elements of $\V$ whose eigenvalues lie in $I$. Given a real valued function $g$ on $I$, we define a map  $G:{\V}_I\rightarrow \V$ as follows. For any $a\in {\V}_I$, let $a=a_1e_1+a_2e_2+\cdots+a_ne_n$ be its spectral decomposition (so that $a_i\in I$ for all $i$). Then, the so-called L\"{o}wner map is given by 
$$G(a):=g(a_1)e_1+g(a_2)e_2+\cdots+g(a_n)e_n.$$
(Our primary examples are: $g(t)=\exp(t)$ on $I=\R$ and $g(t)=\log t$ on $I=(0,\infty)$.)\\
 Now suppose that $g$ is {\it continuously differentiable} on $I$ and let
$$
g^{[1]}(t,s):=\left \{\begin{array}{ll}
\frac{g(t)-g(s)}{t-s} & \mbox{if}\,\,t\neq s,\\
g^\prime(t) & \mbox{if}\,\,t=s.
\end{array}
\right.
$$

In this  setting, Kor\'{a}nyi \cite{koranyi} has proved that $G$ is continuously differentiable at $a$ with (Fr\'{e}chet) derivative  given by 
$$G^{\prime}(a)\,x=\sum_{i\leq j} g^{[1]}(a_i,a_j)\,x_{ij},$$
where 
$x=\sum_{i\leq j}x_{ij}$ is the Peirce decomposition of $x$ with respect to the Jordan frame $\{e_1,e_2,\ldots, e_n\}$. 
By defining  the symmetric matrix 
$$A_a:=\big [ g^{[1]}(a_i,a_j)\big ],$$
we see that $G^{\prime}(a)\,x=A_a\bullet x$. It is shown in \cite{koranyi} that  when $\V$ is simple, 
$G$ is {\it operator monotone} on $\V_I$, that is, 
$$\big [u,v\in \V_I, u\geq v\big ]\Rightarrow G(u)\geq G(v)$$
if and only if $A_a$ is positive semidefinite for all $a\in \V_I$ (that is, for any $\{a_1,a_2,\ldots,a_n\}\subset I$, the matrix
$\big [ g^{[1]}(a_i,a_j)\big ]$ is positive semidefinite).

\gap

We now consider a special case. Let $I=\R$ and $g(t):=\exp(t)$ so that for any $a\in \V$ with spectral representation 
$a=a_1e_1+a_2e_2+\cdots+a_ne_n$, 
$$G(a)=\exp(a)=\sum_{i=1}^{n}\exp(a_i)e_i.$$
 By the above formula of Kor\'{a}nyi,
$G^{\prime}(a)x=A_a\bullet x$, where 
$$A_a=\Big [ \exp^{[1]}(a_i,a_j)\Big ].$$
with $\exp^{[1]}(a_i,a_j)=\frac{\exp(a_i)-\exp(a_j)}{a_i-a_j}$ when $a_i\neq a_j$ and $\exp^{[1]}(a_i,a_j)=\exp(a_i)$ when $a_i= a_j$. 
Let $$c:=\exp(-a/2)=\sum_{i=1}^{n}\exp(-a_i/2)\,e_i$$ and $c_i=\exp(-a_i/2)$ for all $i$.
Then,
$$P_{c}\big (G^{\prime}(a)x\big )= [c_ic_j]\bullet (A_a\bullet x)=\Big [ c_ic_j\exp^{[1]}(a_i,a_j)\Big ]\bullet x=
B\bullet x,$$
where $B=[ b_{ij}]$ with
$$b_{ij}:=\frac{\sinh (\frac{a_i-a_j}{2})}{(\frac{a_i-a_j}{2})}.$$
Let $A=\one_{n\times n}$. As the function $\phi(t)=\frac{t}{\sinh t}$ is positive definite (see \cite{bhatia2}, 5.2.9), the 
real symmetric matrix $\big [\frac{a_{ij}}{b_{ij}}\big ]=\big [\frac{1}{b_{ij}}\big ]$ is a correlation matrix.
We can now
apply Theorem \ref{second main theorem} to  get

\begin{equation}\label{derivative majorization}
x\prec P_{c}\big ( G^{\prime}(a)x\big )\,\,\mbox{and}\,\, ||x||_{sp}\leq ||P_{c}\big (G^{\prime}(a)x\big )||_{sp} \quad \mbox{for all}\,\,x\in \V,
\end{equation}
where the norm inequality follows from (\ref{spectral convexity}). 
It turns out  that the above norm inequality is useful in describing the so-called exponential metric increasing property.

Let
$$\Omega:=\{x\in \V:x>0\}=\exp(\V):=\{\exp(x): x\in \V\}.$$
Given $u,v\in \Omega$, consider a path (piecewise continuously differential curve)
$\gamma:[\alpha,\beta] \rightarrow \Omega$ joining $u$ and $v$, that is, $\gamma(\alpha)=u$ and $\gamma(\beta)=v$. Corresponding to a spectral norm $||\cdot||_{sp}$ on $\V$, we  define the length of this curve to be 
\begin{equation}\label{lenth of gamma}
L(\gamma):=\int_{\alpha}^{\beta}||P_{\gamma(t)^{-1/2}}\big (\gamma^{\prime}(t)\big )||_{sp}\,dt
\end{equation}
and 
$$\delta_{sp}\big (u,v\big ):=\inf\{L(\gamma):\gamma \,\,\mbox{is a path joining}\,\,u\,\,\mbox{and}\,\,v\}.$$
We claim that 
\begin{equation}\label{EMI}
\delta_{sp}\big (u,v\big )\geq ||\log u-\log v||_{sp}\quad (u,v\in \Omega),
\end{equation}
or equivalently,
$$\delta_{sp}\big (\exp(x),\exp(y)\big )\geq ||x-y||_{sp}\quad (x,y\in \V).$$
Motivated by a similar inequality in the setting of ${\cal H}^n$ \cite{bhatia2}, we  will call  this property, the 
{\it exponential metric increasing property}.

To see (\ref{EMI}), we fix $u,v\in \Omega$ and consider a path $\gamma:[\alpha,\beta] \rightarrow \Omega$
joining $u$ and $v$;
correspondingly, define  $\phi:[\alpha,\beta]\rightarrow \V$ by $\phi(t):=\log \gamma(t)$ so that $\gamma(t)=\exp(\phi(t))$.
Now, in (\ref{derivative majorization}) we put $a=\phi(t)$, $c=\gamma(t)^{-1/2}$, and $x=\phi^{\prime}(t)$.
Noting that  $\gamma^{\prime}(t)=\exp^{\prime}(\phi(t))\,\phi^{\prime}(t)$ and 
$ P_{c}\big (G^{\prime}(a)x\big )=P_{\gamma(t)^{-1/2}}\big (\gamma^{\prime}(t)\big )$, we get 
$$L(\gamma)=\int_{\alpha}^{\beta}||P_{\gamma(t)^{-1/2}}\big (\gamma^{\prime}(t)\big )||_{sp}\,dt\geq 
\int_{\alpha}^{\beta}||\phi^{\prime}(t)||_{sp}\,dt\geq ||\log u-\log v||_{sp},$$
where the last item in the above expression is the distance between $\log u$ and $\log v$ in the given spectral norm.
Taking the infimum over all $\gamma$ results in (\ref{EMI}). 

As we see below, in some instances, $\delta_{sp}(u,v)$ can be explicitly described. First, suppose
$u$ and $v$ (in $\Omega$) operator commute so that they have spectral decompositions relative to the same Jordan frame:
$$u=u_1e_1+u_2e_2+\cdots+u_ne_n\quad\mbox{and}\quad v=v_1e_1+v_2e_2+\cdots+v_ne_n.$$
Consider $\phi:[0,1]\rightarrow \V$ defined by $\phi(t)=(1-t)\log u+t\log v$ so that 
$\gamma(t):=\exp(\phi(t))=\sum_{k=1}^{n} u_k^{1-t}v_k^{t}\,e_k$ and 
$\gamma^{\prime}(t)=\sum_{k=1}^{n} (\log v_k-\log u_k)\,u_k^{1-t}v_k^{t}\,e_k$. It follows that 
$$P_{\gamma(t)^{-1/2}}\big (\gamma^{\prime}(t)\big )=\sum_{k=1}^{n} (\log v_k-\log u_k)\,e_k=\log v-\log u.$$
Hence,
$$L(\gamma)=\int_{0}^{1}||P_{\gamma(t)^{-1/2}}\big (\gamma^{\prime}(t)\big )||_{sp}\,dt=||\log v-\log u||_{sp}$$
and by (\ref{EMI}), 
$$\delta_{sp}\big (u,v\big )= ||\log u-\log v||_{sp}=||\log P_{u^{-1/2}}(v)||_{sp}.$$
Now suppose we have a spectral norm $||\cdot||_{sp}$ such that for every  $\gamma$ and $P_r$ ($r>0)$, $L(P_r\circ \gamma)=L(\gamma)$. Then, each $P_r$ is an isometry for the corresponding $\delta_{sp}$, that is,
$$\delta_{sp}\big (u,v\big )=\delta_{sp}\big (P_r(u),P_r(v)\big )\quad (u,v,r>0).$$ 
In this setting, for any $u,v>0$, we can let $r=u^{-1/2}$ so that  $P_r(u)=e$. As $e$ and $P_r(v)$ operate commute, 
$$\delta_{sp}\big (u,v\big )=\delta_{sp}\big (P_r(u),P_r(v))=||\log e-\log P_r(v)||_{sp}=||\log P_{u^{-1/2}}(v)||_{sp}.$$ 
Such a result was proved 
in \cite{lim}, Proposition 2.6 for the spectral norm  $||\cdot||_2$:
$$\delta_{2}(u,v)=\Big ( \sum_{i=1}^{n} \log^2 \lambda_i \Big )^{\frac{1}{2}}=||\log P_{u^{-1/2}}(v)||_2\quad (u,v\in \Omega),$$
where $\lambda_i$s are the eigenvalues of $P_{u^{-1/2}}(v)$.

%%%%%%%%%%%%%%%%%%%%%%%%%%%%%%%%%%%%%%%%%%%%%%%%%%%%%%%%%%%%%%%%%%%%%
\section{Concluding remarks} We end this paper with some remarks on weak majorization. 
Given $p,q\in \Rn$, we say that $p$ is weakly majorized by $q$ and write $p\underset{w}{\prec} q$ if (\ref{weak majorization}) holds for all $1\leq k\leq n$.
For an $n\times n$ complex matrix $A$, let $s(A)$ denote the vector of
singular values of $A$ written in the decreasing order.
Then, for $A,B\in {\cal M}^n$, by the Fan domination Theorem (see \cite{bhatia1}, Theorem IV.2.2),
$s(A)\underset{w}{\prec}s(B)$ if and only if $|||A|||\leq |||B|||$ for all unitary invariant norms $|||\cdot|||$ on ${\cal M}^n$. Equivalently, see (\cite{bhatia1}, Theorem IV.2.1), $s(A)\underset{w}{\prec}s(B)$ if and only if
$\Phi(s(A))\leq \Phi(s(B))$ for all symmetric gauge functions $\Phi$ on $\Rn$. (Recall that a symmetric gauge function on $\Rn$  is a permutation invariant norm on $\Rn$ that is additionally sign invariant, see \cite{bhatia1}, p. 86.)
On our Euclidean Jordan algebra $\V$, we  define weak majorization by $x\underset{w}{\prec}y\Leftrightarrow \lambda(x)\underset{w}{\prec}\lambda(y)$ in $\Rn$.
\\
Now, for  $x,y\in \V$, consider the following: 
\begin{itemize}
\item [(a)] $x\prec y$.
\item [(b)] $||x||_{sp}\leq ||y||_{sp}$ for all spectral norms $||\cdot||_{sp}$.
\item [(c)] $|x|\underset{w}{\prec}|y|.$
\end{itemize}
Then $(a)\Rightarrow (b)\Rightarrow (c)$.
Since the implication $(a)\Rightarrow (b)$ follows from (\ref{spectral convexity}), we justify the implication $(b)\Rightarrow (c)$. Assume $(b)$ holds and write $||x||_{sp}=||\lambda(x)||_{perm}$, etc., where $||\cdot||_{perm}$ is a permutation invariant norm on $\Rn$. Specializing, we can let $||\cdot||_{perm}$ be a symmetric gauge function $\Phi(\cdot)$. Hence, $(b)$ implies 
$\Phi(\lambda(x))\leq \Phi(\lambda(y))$, or equivalently, $\Phi(\lambda(|x|))\leq \Phi(\lambda(|y|))$. This, by the above 
mentioned results  on singular values,  gives $(c)$.
We remark that $(c)$ may not imply $(b)$ as there are permutation invariant norms on $\Rn$ that are not symmetric gauge functions. 
\\
Following the theme of the paper (on pointwise majorization inequalities arising from Schur products), we may ask if pointwise weak majorization inequalities of the form $A\bullet x\underset{w}{\prec} B\bullet x$ can be proved. Since requiring this inequality for all $x$ results in $\tr(A\bullet x)=\tr(B\bullet x)$ for all $x$, we are back to our majorization inequalities. So, it becomes necessary to restrict $x$ by imposing a condition such as $x\geq 0$.
Now, we may ask: For which matrices $A\in \Sn$, do we have $A\bullet x\underset{w}{\prec} x$ for all $x\geq 0$? In a private communication \cite{jeong-private}, Jeong studies the concept of doubly substochastic map in the setting of Euclidean Jordan algebras and 
shows that when $A\in \Sn$ is positive semidefinite,  the inequality  $A\bullet x\underset{w}{\prec} x$ holds 
for all $x\geq 0$ if and  only all diagonal entries of $A$ are less than or equal to one.
%%%%%%%%%%%%%% %%%%%%%%%%%%%%%%%%%%%%%%%%%%%%%%%%%%%%%%%%%%%%%%%%%%%%%%%%%%%%%%%%%%%%%%%%%

\end{document}